\documentclass{amsart}
\usepackage{amsmath,amssymb,amscd,cite,geometry,pgfplots,enumerate}
\usepackage[colorlinks,citecolor=blue,linkcolor=blue,urlcolor=blue]{hyperref}
\newtheorem{theorem}{Theorem}[section]

\newtheorem{corollary}[theorem]{Corollary}

\newtheorem{remark}[theorem]{Remark}
\numberwithin{equation}{section}
\DeclareUnicodeCharacter{2212}{-}
\allowdisplaybreaks[4]
\pgfplotsset{compat=1.17}

\begin{document}

\title{Asymptotic Expansions of The Traces of the Thermoelastic Operators}

\author{Genqian Liu}
\address{School of Mathematics and Statistics, Beijing Institute of Technology, Beijing 100081, China}
\email{liugqz@bit.edu.cn}

\author{Xiaoming Tan}
\address{School of Mathematics and Statistics, Beijing Institute of Technology, Beijing 100081, China}
\email{xtan@bit.edu.cn}


\begin{abstract}
    We obtain the asymptotic expansions of the traces of the thermoelastic operators with the Dirichlet and Neumann boundary conditions on a Riemannian manifold, and give an effective method to calculate all the coefficients of the asymptotic expansions. These coefficients provide precise geometric information. In particular, we explicitly calculate the first two coefficients concerning the volumes of the manifold and its boundary. As an application, by combining our results with the isoperimetric inequality we show that an $n$-dimensional geodesic ball is uniquely determined up to isometry by its thermoelastic spectrum among all bounded thermoelastic bodies with boundary.
\end{abstract}

\maketitle

\section{Introduction}

\addvspace{5mm}

Let $\Omega$ be a smooth compact Riemannian manifold of dimension $n$ with smooth boundary $\partial \Omega$. We consider $\Omega$ as a homogeneous, isotropic, thermoelastic body. Assume that the Lam\'{e} coefficients $\lambda$, $\mu$ and the heat conduction coefficient $\alpha$ of the thermoelastic body are constants which satisfy $\mu > 0$, $\lambda + \mu \geqslant 0$ and $\alpha>0$. We denote by $\operatorname{grad},\, \operatorname{div},\, \Delta_{g},\, \Delta_{B}$ and $\operatorname{Ric}$, respectively, the gradient operator, the divergence operator, the Laplace--Beltrami operator, the Bochner Laplacian and the Ricci tensor. We define the thermoelastic operator
\begin{align}\label{1.2}
    \mathcal{L}_g :=
    \begin{pmatrix}
        \mu \Delta_{B} + (\lambda + \mu)\operatorname{grad}\operatorname{div} + \mu \operatorname{Ric} + \rho\omega^2 & -\beta\operatorname{grad} \\
        i\omega\theta_{0}\beta\operatorname{div} & \alpha\Delta_{g} + i\omega \gamma
    \end{pmatrix},
\end{align}
where the coefficient $\beta$ depends on $\lambda,\,\mu$ and the linear expansion coefficient, $\gamma$ is the specific heat per unit volume, $\theta_{0}$ is the reference temperature, $\rho$ is the density of the thermoelastic body, $\omega$ is the angular frequency and $i = \sqrt{-1}$. We denote by $\textbf{\textit{u}}(x) \in (C^{\infty}(\Omega))^{n}$ and $\theta(x) \in C^{\infty}(\Omega)$ the displacement and the temperature variation, respectively. By the methods of \cite{Kupr80,Liu_NL21}, we can write the steady thermoelastic equation as follows
\begin{align}\label{1.1}
    \mathcal{L}_g \textbf{\textit{U}} = 0 \quad \text{in}\ \Omega,
\end{align}
where $\textbf{\textit{U}} = (\textbf{\textit{u}},\theta)^t$, the superscript $t$ denotes the transpose. Equation \eqref{1.1} is an extension of the classical elastic equation. In particular, when $\Omega$ is a bounded domain in $\mathbb{R}^n$ and the temperature of the medium is not taken into consideration, equation \eqref{1.1} reduces to the classical elastic equation (see pp.\,35--36 of \cite{Kupr80}).

In the case of classical elasticity, assume that the temperature of the medium is the same at all points and is not changed during deformation. Nevertheless, in reality, most materials are not pure elastic, a temperature change in the elastic body will generate additional strain and stress. Therefore, deformation is followed by temperature variation and, conversely, temperature variation is followed by deformation of the thermoelastic body due to thermal expansion. Even if there is no external stresses or mass forces, some heat sources change the temperature of the thermoelastic body, there will be deformation. The significance of the equation \eqref{1.1} consists not only in the important independent role of this state, frequently occurring in engineering, but also in the fact that their investigation opens up the way to the study the general type (see p.\,528 of \cite{Kupr80}).

\addvspace{2mm}

In this paper, we study the asymptotic expansions of the taces of the thermoelastic operators on $\Omega$. We denote by $\mathcal{L}_g^-$ and $\mathcal{L}_g^+$ the thermoelastic operators with the Dirichlet and Neumann boundary conditions, respectively. Here the Neumann boundary condition is defined by, for any $\textbf{\textit{U}} \in (C^1(\Omega))^{n+1}$,
\begin{align*}
    N (\textbf{\textit{U}}) :=
    \begin{pmatrix}
        \lambda\nu \operatorname{div} + \mu\nu S & -\beta \nu \\
        0 & \alpha\partial_\nu
    \end{pmatrix}
    \textbf{\textit{U}} \quad \text{on}\ \partial \Omega,
\end{align*}
where $\nu$ is the outward unit normal vector to $\partial \Omega$ and $S$ is the stress tensor (also called the deformation tensor) of type $(1,1)$ defined by $S\textbf{\textit{u}}:=\nabla \textbf{\textit{u}} + \nabla \textbf{\textit{u}}^t$ for any $\textbf{\textit{u}} \in (C^1(\Omega))^n$ (see p.~562 of \cite{Taylor11.3}).  According to the theory of elliptic operators, $\mathcal{L}_g^-$ (respectively, $\mathcal{L}_g^+$) is an unbounded, self-adjoint and negative (respectively, non-positive) operator in $(H_0^1(\Omega))^{n+1}$ (respectively, $(H^1(\Omega))^{n+1}$), it follows from \cite{Avra06,Hook90,Laptev97,Liu_NL21,Pleijel39} that there exists a sequence $\{\tau^{-}_k\}_{k \geqslant 1}$ (respectively, $\{\tau^{+}_k\}_{k \geqslant 1}$) such that the following eigenvalue problem
\begin{align}\label{1.6}
    \begin{cases}
        \mathcal{L}_g \textbf{\textit{U}}^{-}_k + \tau^{-}_k \textbf{\textit{U}}^{-}_k = 0 & \text{in}\ \Omega, \\
        \textbf{\textit{U}}^{-}_k = 0 & \text{on}\ \partial \Omega
    \end{cases}
\end{align}
(respectively,
\begin{align}\label{1.7}
    \begin{cases}
        \mathcal{L}_g \textbf{\textit{U}}^{+}_k + \tau^{+}_k \textbf{\textit{U}}^{+}_k = 0 & \text{in}\ \Omega, \\
        N(\textbf{\textit{U}}^{+}_k) = 0 & \text{on}\ \partial \Omega)
    \end{cases}
\end{align}
has a discrete spectrum $0 < \tau^-_{1} < \tau^-_{2} \leqslant \cdots \leqslant \tau^-_{k} \leqslant \cdots \to +\infty$
(respectively, $0 \leqslant \tau^+_{1} < \tau^+_{2} \leqslant \cdots \leqslant \tau^+_{k} \leqslant \cdots \to +\infty$), where $\textbf{\textit{U}}^-_k\in (H_0^1(\Omega))^{n+1}$ (respectively, $\textbf{\textit{U}}^+_k  \in (H^1(\Omega))^{n+1}$) is the eigenvector corresponding to the Dirichlet eigenvalue $\tau^-_{k}$ (respectively, Neumann eigenvalue $\tau^+_{k}$).

\addvspace{2mm}

Let us point out that the thermoelastic operator is a non-Laplace type operator. Contrary to the Laplace type operators, there are no systematic effective methods to explicitly calculate the coefficients of the asymptotic expansions of the traces for non-Laplace type operators. In 1971, Greiner \cite{Greiner71} indicated that {\it “the problem of interpreting these coefficients geometrically remains open”}. In this sense, it is just beginning that the study of geometric aspects of spectral asymptotics for non-Laplace type operators, the corresponding methodology is still underdeveloped in comparison with the theory of the Laplace type operator. Thus, {\it the geometric aspect of the spectral asymptotics of non-Laplace type operators remains an open problem} (see \cite{Avra06}).

Due to the fact that the Riemannian structure on a manifold is determined by a Laplace type operator, the systematic explicit calculations of coefficients of heat kernels for Laplace type operators are now well understood, see Gilkey \cite{Gilkey75} and many others (see \cite{AGMT10,BGV92,Gilkey79,Gilkey95,Grubb86,Grubb09,Kirsten01,Liu11,Vassilevich03} and references therein). For the classical boundary conditions, like Dirichlet, Neumann, Robin, and mixed combination thereof on vector bundles, the coefficients of the asymptotic expansions have been explicitly computed up to the first five terms (see, for example, \cite{BranGilk90,BPKV99,Kirs98}). For other types of operators, the first author of this paper gave the first two coefficients of the asymptotic expansions of the traces of the Navier--Lam\'{e} operator \cite{Liu_NL21} and the Stokes operator \cite{Liu_S21}. The first author of this paper also gave the first four coefficients of the asymptotic expansions of heat traces for the classical Dirichlet-to-Neumann map \cite{Liu15} (see Polterovich and Sher \cite{PoltSher15} for the first three coefficients) and the polyharmonic Steklov operator \cite{Liu14}, and gave the first two coefficients of the asymptotic expansions of the heat traces for elastic Dirichlet-to-Neumann map \cite{Liu19} as well. Recently, the authors of this paper have obtained the first four coefficients of the asymptotic expansion of heat trace for the magnetic Dirichlet-to-Neumann map associated with the magnetic Schr\"{o}dinger operator \cite{LiuTan21}. To the best of our knowledge, it has not previously been systematically applied in the context of the thermoelastic operators with the Dirichlet and Neumann boundary conditions.

\addvspace{2mm}

The main result of this paper is the following theorem.
\begin{theorem}\label{thm1.1}
    Let $\Omega$ be a smooth compact Riemannian manifold of dimension $n$ with smooth boundary $\partial \Omega$. Assume that $\alpha>0$, $\mu > 0$ and $\lambda+\mu \geqslant 0$. Let $\mathcal{L}_g^-$ and $\mathcal{L}_g^+$ be the thermoelastic operators with the Dirichlet and Neumann boundary conditions, respectively, and $\{\tau^{\mp}_k\}_{k\geqslant 1}$ satisfy the eigenvalue equations $\mathcal{L}_g^{\mp} \textbf{\textit{U}}^{\mp}_k + \tau^{\mp}_k \textbf{\textit{U}}^{\mp}_k = 0$. Then
    \begin{align*}
        \sum_{k=1}^{\infty} e^{-t \tau^{\mp}_k} 
        = \operatorname{Tr} (e^{t \mathcal{L}_g^{\mp}}) 
        =t^{-n/2} \bigl(a_0+a_1^{\mp}t^{1/2}+\cdots+a_n^{\mp}t^{n/2}+O(t^{(n+1)/2})\bigr)\quad \text{as}\ t \to 0^+.
    \end{align*}
    In particular,
    \begin{align*}
        a_0 &= \frac{\operatorname{vol}(\Omega)}{(4\pi)^{n/2}} \biggl(\frac{n-1}{\mu^{n/2}} + \frac{1}{(\lambda+2\mu)^{n/2}} + \frac{1}{\alpha^{n/2}}\biggr), \\[1mm]
        a_1^{\mp} &= \mp \frac{1}{4}\frac{\operatorname{vol}(\partial \Omega)}{(4\pi)^{(n-1)/2}} \biggr(\frac{n-1}{\mu^{(n-1)/2}} + \frac{1}{(\lambda+2\mu)^{(n-1)/2}} + \frac{1}{\alpha^{(n-1)/2}}\biggr),
    \end{align*}
    where $\operatorname{vol}(\Omega)$ and $\operatorname{vol}(\partial \Omega)$ denote the $n$-dimensional volume of $\Omega$ and the $(n-1)$-dimensional volume of $\partial \Omega$, respectively.
\end{theorem}

\addvspace{2mm}

\begin{remark}
    Theorem {\rm\ref{thm1.1}} shows that both the volume $\operatorname{vol}(\Omega)$ and $\operatorname{vol}(\partial \Omega)$ are spectral invariants, they can also be obtained by the thermoelastic eigenvalues with respect to the Dirichlet or Neumann boundary condition. This gives an answer to an interesting and open problem mentioned in {\rm\cite{Avra06,Greiner71,Kac66}}. Roughly speaking, one can “hear” the volume of the domain and the surface area of its boundary by “hearing” all the pitches of the vibration of a thermoelastic body. This result is a generalization of the corresponding result for the elastic Navier--Lam\'{e} operator $($see {\rm\cite{Liu_NL21}}$)$.
\end{remark}
\begin{remark}
    By applying the Tauberian theorem $($see Theorem $15.3$ on p.\,$30$ of {\rm\cite{Kore04}} or p.\,$107$ of {\rm\cite{Taylor11.2}}$)$, we can easily obtain the Weyl-type law for the counting function $N^{\mp}(\tau):= \#\{k|\tau_k^{\mp} \leqslant \tau\}$ of thermoelastic Steklov eigenvalues:
    \begin{align*}
        N^{\mp}(\tau) = & \frac{\operatorname{vol}(\Omega)}{(4\pi)^{n/2}\Gamma(1+n/2)} \biggl(\frac{n-1}{\mu^{n/2}} + \frac{1}{(\lambda+2\mu)^{n/2}} + \frac{1}{\alpha^{n/2}}\biggr) \tau^{n/2} + o(\tau^{n/2}) \quad \text{as}\ \tau \to +\infty.
    \end{align*}
\end{remark}

\addvspace{2mm}

As an application of Theorem \ref{thm1.1}, we have the following spectral rigidity result:
\begin{corollary}\label{cor1.4}
    Let $\mathcal{M}$ be a Riemannian manifold of dimension $n$, $\Omega \subset \mathcal{M}$ be a compact domain with smooth boundary and $B^n \subset \mathcal{M}$ be an $n$-dimensional geodesic ball. Asuume that the following two conditions hold:
    \begin{enumerate}[$(a)$]
        \item The thermoelastic spectrum of $\Omega$ with respect to the Dirichlet or Neumann boundary condition is equal to that of $B^n$;
        \item The isoperimetric inequality holds on $\mathcal{M}$, $($i.e., let $\tilde{\Omega} \subset \mathcal{M}$ be a compact subset with boundary. If $\operatorname{vol}(\tilde{\Omega}) = \operatorname{vol}(B^n)$, then $\operatorname{vol}(\partial\tilde{\Omega}) \geqslant \operatorname{vol}(\partial B^n)$, equality holds if and only if $\tilde{\Omega}$ is isometric to $B^n$$)$.
    \end{enumerate}
    Then $\Omega$ is isometric to $B^n$.
\end{corollary}

\begin{remark}
    There are many cases such that condition $(b)$ holds, we only list some of them as follows:
    \begin{enumerate}[{\rm (i)}]
        \item $\mathcal{M}$ is a space form, $($i.e., a complete Riemannian manifold with constant sectional curvature$)$, $($see {\rm \cite{Osserman78,Schmidt1,Schmidt2,Schmidt3,Schmidt4,Schmidt5}}$)$;
        \item $\mathcal{M} \subset \mathbb{R}^{n+m}\ (m=1,2)$ is a compact $n$-dimensional minimal submanifold with boundary $($see {\rm \cite{Brendle1}}$)$;
        \item $\mathcal{M}$ is a complete noncompact manifold with nonnegative Ricci curvature $($see {\rm \cite{Brendle2}}$)$;
        \item $\mathcal{M}$ is a compact $n$-dimensional minimal submanifold with boundary in a complete noncompact manifold of dimension $n+m\ (m=1,2)$ with nonnegative sectional curvature $($see {\rm \cite{Brendle2}}$)$;
        \item $\mathcal{M}$ is a Cartan--Hadamard manifold of dimension $n$ $($i.e., a complete, simply connected Riemannian manifold with nonpositive sectional curvature $K)$ for $n=2$ $($see {\rm \cite{Weil79,BeckRado33}}$)$, $n=3$ $($see {\rm \cite{Kleiner92}}$)$, $n=4$ $($see {\rm \cite{Croke84}}$)$. Generally, the isoperimetric inequality also holds when $K\leqslant k \leqslant 0$ for $n=2$ $($see {\rm\cite{Bol41}}$)$, $n=3$ $($see {\rm\cite{Kleiner92}}$)$.
    \end{enumerate}
\end{remark}

\addvspace{2mm}

The main ideas of this paper are as follows. Firstly, we denote by $(e^{t\mathcal{L}_g^{\mp}})_{t\geqslant 0}$ the parabolic semigroups generated by $\mathcal{L}_g^{\mp}$. More precisely, $\textbf{\textit{V}}^{\mp}(t,x)=e^{t\mathcal{L}_g^{\mp}}\textbf{\textit{V}}(x)$ solve the following initial-boundary problems, respectively,
\begin{align*}
    \begin{cases}
        \frac{\partial \textbf{\textit{V}}^{-}}{\partial t} - \mathcal{L}_g \textbf{\textit{V}}^{-} = 0 & \text{in}\ (0,+\infty)\times\Omega,\\
        \textbf{\textit{V}}^{-} = 0 \quad & \text{on}\ (0,+\infty)\times \partial\Omega,\\
        \textbf{\textit{V}}^{-}(0,x) = \textbf{\textit{V}}_0 & \text{on}\ \{0\}\times\Omega
    \end{cases}
\end{align*}
and
\begin{align*}
    \begin{cases}
        \frac{\partial \textbf{\textit{V}}^{+}}{\partial t} - \mathcal{L}_g \textbf{\textit{V}}^{+} = 0  & \text{in}\ (0,+\infty)\times\Omega,\\
        N(\textbf{\textit{V}}^{+}) = 0 \quad& \text{on}\ (0,+\infty)\times \partial\Omega,\\
        \textbf{\textit{V}}^{+}(0,x) = \textbf{\textit{V}}_0 & \text{on}\ \{0\}\times\Omega.
    \end{cases}
\end{align*}
Let $\{\textbf{\textit{U}}^{\mp}_k\}_{k \geqslant 1}$ be the orthonormal eigenvectors corresponding to the eigenvalues $\{\tau^{\mp}_k\}_{k \geqslant 1}$, then the integral kernels $\textbf{\textit{K}}^{\mp}(t,x,y)$ are given by
\begin{align*}
    \textbf{\textit{K}}^{\mp}(t,x,y) = \sum_{k=1}^{\infty} e^{-t\tau^{\mp}_k} \textbf{\textit{U}}^{\mp}_k(x) \otimes \textbf{\textit{U}}^{\mp}_k(y).
\end{align*}
Thus,
\begin{align*}
    \int_{\Omega} \operatorname{Tr} \textbf{\textit{K}}^{\mp}(t,x,y) \,dV = \sum_{k=1}^{\infty} e^{-t\tau^{\mp}_k}
\end{align*}
are spectral invariants.

Furthermore, to investigate the geometric aspects of the spectrum and the asymptotic expansion of the integrals, we combine calculus of symbols (see \cite{Gilkey75,Grubb86,Seeley69}) and “method of images” (see \cite{Liu_NL21,McKeanSinger67}). Let $M = \Omega \cup \partial \Omega \cup \Omega^{*}$ be the closed double of $\Omega$, and $L_g$ the double to $M$ of the operator $\mathcal{L}_g$. Then $L_g$ generates a strongly continuous semigroup $(e^{tL_g})_{t\geqslant 0}$ on $L^2(M)$ with integral kernel $\textbf{\textit{K}}(t,x,y)$. Clearly, $\textbf{\textit{K}}^{\mp}(t,x,y) = \textbf{\textit{K}}(t,x,y) \mp \textbf{\textit{K}}(t,x,y^*)$ for $x,y \in \bar{\Omega}$, where $y^*$ is the double of $y\in \Omega$. Since
\begin{align*}
    e^{tL_g} f(x) = \frac{1}{2\pi i}\int_{\mathcal{C}} e^{-t\tau}(\tau I+L_g)^{-1} f(x) \,d\tau,
\end{align*}
where $\mathcal{C}$ is a suitable curve in the complex plane in the positive direction around the spectrum of $L_g$. It follows that
\begin{align*}
    \textbf{\textit{K}}(t,x,y) = \frac{1}{(2\pi)^{n}} \int_{\mathbb{R}^n} e^{i (x-y)\cdot\xi} 
    \bigg(
        \frac{1}{2\pi i}\int_{\mathcal{C}} e^{-t\tau} \sigma\big((\tau I+L_g)^{-1}\big) \,d\tau
    \bigg) \,d\xi,
\end{align*}
where $\sigma\big((\tau I+L_g)^{-1}\big)$ is the full symbol of $(\tau I+L_g)^{-1}$. In particular, we get $\operatorname{Tr} \textbf{\textit{K}}(t,x,x)$ and $\operatorname{Tr} \textbf{\textit{K}}(t,x,x^*)$ for any $x\in \Omega$. Finally, we explicitly calculate the first two coefficients $a_0$ and $a_1$ of the asymptotic expansion
\begin{align*}
    \int_{\Omega} \operatorname{Tr}\textbf{\textit{K}}^{\mp}(t,x,x)\,dV
    = a_0t^{-n/2} \mp a_1t^{-(n-1)/2}+O(t^{1-n/2}) \quad \text{as}\ t \to 0^+,
\end{align*}
where $a_0$ and $a_1$ are given in Theorem \ref{thm1.1}.

This paper is organized as follows. In Section \ref{s2} we give the expression of the thermoelastic operator in local coordinates and derive the full symbol of the resolvent operator $(\tau I + \mathcal{L}_g)^{-1}$. Section \ref{s3} is devoted to prove Theorem \ref{thm1.1} and Corollary \ref{cor1.4}.

\addvspace{10mm}

\section{Full Symbol of The Resolvent Operator}\label{s2}

\addvspace{5mm}

In the local coordinates $\{x_i\}_{i=1}^n$, we denote by $\bigl\{\frac{\partial}{\partial x_i}\bigr\}_{i=1}^n$ a natural basis for the tangent space $T_x \Omega$ at the point $x \in \Omega$. In what follows, we will use the Einstein summation convention, the Roman indices run from 1 to $n$, unless otherwise specified. Then the Riemannian metric is given by $g = g_{ij} \,dx_i\, dx_j$. Let $\nabla$ be the Levi-Civita connection on $\Omega$, $\nabla_i = \nabla_{\frac{\partial}{\partial x_i}}$ and $\nabla^i= g^{ij} \nabla_j$. The gradient operator is denoted by
\begin{equation}\label{2.1}
    \operatorname{grad} f = \nabla^i f \frac{\partial}{\partial x_i}
    = g^{ij} \frac{\partial f}{\partial x_i} \frac{\partial}{\partial x_j},\quad f \in C^{\infty}(\Omega),
\end{equation}
where $(g^{ij}) = (g_{ij})^{-1}$. The divergence operator is denoted by
\begin{equation}\label{2.2}
    \operatorname{div} \textbf{\textit{u}} = \nabla_i u^i = \frac{\partial u^i}{\partial x_i} + \Gamma^j_{ij} u^i,\quad \textbf{\textit{u}}=u^i \frac{\partial}{\partial x_i}\in \mathfrak{X}  (\Omega),
\end{equation}
where the Christoffel symbols $\Gamma^{k}_{ij} = \frac{1}{2} g^{kl} \bigl(\frac{\partial g_{jl}}{\partial x_i} + \frac{\partial g_{il}}{\partial x_j} - \frac{\partial g_{ij}}{\partial x_l}\bigr)$. It follows from \eqref{2.1} and \eqref{2.2} that
\begin{align}\label{2.3}
    \operatorname{grad}\operatorname{div} \textbf{\textit{u}} =  g^{ij}
    \Bigl( \frac{\partial^2 u^k}{\partial x_j \partial x_k} +  \Gamma^l_{kl}\frac{\partial u^k}{\partial x_j} +  \frac{\partial \Gamma^l_{kl}}{\partial x_j} u^k \Bigr) \frac{\partial}{\partial x_i}.
\end{align}

Recall that the components of the curvature tensor
\begin{align}\label{2.4}
    R^{l}_{ijk} 
    = \frac{\partial \Gamma^{l}_{jk}}{\partial x_i} - \frac{\partial \Gamma^{l}_{ik}}{\partial x_j} + \Gamma^{h}_{jk} \Gamma^{l}_{ih} - \Gamma^{h}_{ik} \Gamma^{l}_{jh},
\end{align}
and $R_{ijkl} = g_{lm} R^{m}_{ijk}$, the components of the the Ricci tensor
\begin{align}\label{2.5}
    R_{ij} = g^{kl} R_{iklj}.
\end{align}
The Bochner Laplacian is given by
\begin{align*}
    \Delta_{B}\textbf{\textit{u}} =(\nabla^j \nabla_j u^i)\frac{\partial}{\partial x_i} ,\quad \textbf{\textit{u}}=u^i \frac{\partial}{\partial x_i}\in \mathfrak{X}  (\Omega).
\end{align*}
Then combining \eqref{2.4} and \eqref{2.5}, we compute
\begin{align}\label{2.6}
    (\Delta_{B}\textbf{\textit{u}})^i 
    & = \Delta_g u^i + g^{kl} 
    \Bigl(
        2\Gamma^i_{jk}\frac{\partial u^j}{\partial x_l} + \frac{\partial \Gamma^i_{jk}}{\partial x_l}u^j + (\Gamma^i_{hl}\Gamma^h_{jk} - \Gamma^h_{kl}\Gamma^i_{jh}) u^j
    \Bigr) \\
    & = \Delta_g u^i - \operatorname{Ric}(\textbf{\textit{u}})^i + g^{kl} \Bigl( 2\Gamma^i_{jk} \frac{\partial u^j}{\partial x_l} + \frac{\partial \Gamma^i_{kl}}{\partial x_j} u^j \Bigr), \notag
\end{align}
where $\operatorname{Ric}(\textbf{\textit{u}})^i = g^{ij}R_{jk} u^k$ and the Laplace--Beltrami operator is given by
\begin{equation}\label{2.7}
    \Delta_{g} f
    = g^{ij}\Bigl(\frac{\partial^2 f}{\partial x_i \partial x_j} - \Gamma^k_{ij}\frac{\partial f}{\partial x_k}\Bigr), \quad f \in C^{\infty}(\Omega).
\end{equation}

\addvspace{2mm}

For the sake of simplicity, we denote by $I_n$ the $n\times n$ identity matrix and
\begin{align*}
    \begin{pmatrix}
        (a^j_k) & (b^j) \\
        (c_k) & d
    \end{pmatrix}
    :=
    \begin{pmatrix}
        a^1_1 & \dots & a^1_n & b^1 \\
        \vdots & \ddots & \vdots & \vdots \\
        a^n_1 & \dots & a^n_n & b^n \\
        c_1 & \dots & c_n & d \\
    \end{pmatrix}.
\end{align*}
According to \eqref{2.1}, \eqref{2.2}, \eqref{2.3}, \eqref{2.6} and \eqref{2.7} (cf. \cite{Liu_NL21}), we can write \eqref{1.2} in local coordinates as follows:
\begin{align*}
    \mathcal{L}_g &=
    \begin{pmatrix}
            \mu I_n  & 0 \\
            0 & \alpha
    \end{pmatrix}
    g^{ml}\frac{\partial^2 }{\partial x_m \partial x_l}
    + (\lambda+\mu)
    \begin{pmatrix}
            \big(g^{jm}\frac{\partial^2 }{\partial x_k \partial x_m}\big) & 0 \\
            0 & 0
    \end{pmatrix}
    -
    \begin{pmatrix}
            \mu I_n  & 0 \\
            0 & \alpha
    \end{pmatrix}
    g^{ml}\Gamma^s_{ml}\frac{\partial }{\partial x_s} \\
    &\quad + 2\mu
    \begin{pmatrix}
        \big(g^{ml}\Gamma^j_{mk}\frac{\partial }{\partial x_l}\big) & 0 \\
        0 & 0
    \end{pmatrix}
    + (\lambda+\mu)
    \begin{pmatrix}
        \big(g^{jm}\Gamma^l_{kl}\frac{\partial }{\partial x_m}\big) & 0 \\
        0 & 0
    \end{pmatrix}
    + \beta
    \begin{pmatrix}
        0 & -\big(g^{jk}\frac{\partial }{\partial x_k}\big) \\
        i\omega\theta_0 \big(\frac{\partial }{\partial x_k}\big) & 0
    \end{pmatrix} \\
    &\quad+ (\lambda+\mu)
    \begin{pmatrix}
        \big(g^{jm}\frac{\partial \Gamma^l_{kl}}{\partial x_m}\big) & 0 \\
        0 & 0
    \end{pmatrix}
    + \mu
    \begin{pmatrix}
        \big(g^{ml}\frac{\partial \Gamma^j_{ml}}{\partial x_k}\big) & 0 \\
        0 & 0
    \end{pmatrix}
    +
    \begin{pmatrix}
        \rho\omega^2 I_n & 0 \\
        i\omega\beta\theta_0 (\Gamma^l_{kl}) & i\omega \gamma
    \end{pmatrix}.
\end{align*}
Moreover, we have
\begin{align*}
    \mathcal{L}_g \textbf{\textit{U}}(x) = \frac{1}{(2\pi)^n}\int_{\mathbb{R}^n}e^{ix\cdot \xi}\textbf{\textit{C}}_g(x,\xi)\hat{\textbf{\textit{U}}}(\xi)\,d\xi,
\end{align*}
where
\begin{align*}
    \textbf{\textit{C}}_g(x,\xi) &=
    -|\xi|^2
    \begin{pmatrix}
            \mu I_n  & 0 \\
            0 & \alpha
    \end{pmatrix}
    - (\lambda+\mu)
    \begin{pmatrix}
            (g^{jm}\xi_k\xi_m) & 0 \\
            0 & 0
    \end{pmatrix}
    - ig^{ml}\Gamma^s_{ml}\xi_s
    \begin{pmatrix}
            \mu I_n  & 0 \\
            0 & \alpha
    \end{pmatrix} \\
    &\quad + 2i\mu
    \begin{pmatrix}
        (g^{ml}\Gamma^j_{mk}\xi_l) & 0 \\
        0 & 0
    \end{pmatrix}
    + i(\lambda+\mu)
    \begin{pmatrix}
        (g^{jm}\Gamma^l_{kl}\xi_m) & 0 \\
        0 & 0
    \end{pmatrix}
    - \beta
    \begin{pmatrix}
        0 & i(g^{jk}\xi_k) \\
        \omega\theta_0 (\xi_k) & 0
    \end{pmatrix} \\
    &\quad+ (\lambda+\mu)
    \begin{pmatrix}
        \big(g^{jm}\frac{\partial \Gamma^l_{kl}}{\partial x_m}\big) & 0 \\
        0 & 0
    \end{pmatrix}
    + \mu
    \begin{pmatrix}
        \big(g^{ml}\frac{\partial \Gamma^j_{ml}}{\partial x_k}\big) & 0 \\
        0 & 0
    \end{pmatrix}
    +
    \begin{pmatrix}
        \rho\omega^2 I_n & 0 \\
        i\omega\beta\theta_0 (\Gamma^l_{kl}) & i\omega \gamma
    \end{pmatrix},
\end{align*}
and $|\xi|^2 = g^{ij}\xi_i\xi_j$. For each $\tau \in \mathbb{C}$, we denote $\tau I_n + \textbf{\textit{C}}_g = \textbf{\textit{c}}_2 + \textbf{\textit{c}}_1 + \textbf{\textit{c}}_0$, where
\begin{align}
    \textbf{\textit{c}}_2 & =
    \begin{pmatrix}
            (\tau - \mu|\xi|^2) I_n  & 0 \\
            0 & \tau - \alpha|\xi|^2
    \end{pmatrix}
    - (\lambda+\mu)
    \begin{pmatrix}
            (g^{jm}\xi_k\xi_m) & 0 \\
            0 & 0
    \end{pmatrix}, \label{3.1}\\
    \textbf{\textit{c}}_1 & =
    - ig^{ml}\Gamma^s_{ml}\xi_s
    \begin{pmatrix}
            \mu I_n  & 0 \\
            0 & \alpha
    \end{pmatrix}
    + 2i\mu
    \begin{pmatrix}
        (g^{ml}\Gamma^j_{mk}\xi_l) & 0 \\
        0 & 0
    \end{pmatrix}
    + i(\lambda+\mu)
    \begin{pmatrix}
        (g^{jm}\Gamma^l_{kl}\xi_m) & 0 \\
        0 & 0
    \end{pmatrix} \label{3.2}\\
    &\quad - \beta
    \begin{pmatrix}
        0 & i(g^{jk}\xi_k) \\
        \omega\theta_0 (\xi_k) & 0
    \end{pmatrix}, \notag\\
    \textbf{\textit{c}}_0 & =
    (\lambda+\mu)
    \begin{pmatrix}
        \big(g^{jm}\frac{\partial \Gamma^l_{kl}}{\partial x_m}\big) & 0 \\
        0 & 0
    \end{pmatrix}
    + \mu
    \begin{pmatrix}
        \big(g^{ml}\frac{\partial \Gamma^j_{ml}}{\partial x_k}\big) & 0 \\
        0 & 0
    \end{pmatrix}
    +
    \begin{pmatrix}
        \rho\omega^2 I_n & 0 \\
        i\omega\beta\theta_0 (\Gamma^l_{kl}) & i\omega \gamma
    \end{pmatrix}.\label{3.3}
\end{align}

\addvspace{2mm}

Now we calculate the full symbol of the operator $(\tau I + \mathcal{L}_g)^{-1}$. Let $B(x,\xi,\tau)$ be a two-sided parametrix for $\tau I + \mathcal{L}_g$, that is, $B(x,\xi,\tau)$ is a pseudodifferential operator of order $-2$ with parameter $\tau$ for which
\begin{align*}
    & B(\tau)(\tau I + \mathcal{L}_g) = I \mod OPS^{-\infty}, \\
    & (\tau I + \mathcal{L}_g)B(\tau) = I \mod OPS^{-\infty}.
\end{align*}
Recall that if $P_1$ and $P_2$ are two pseudodifferential operators with full symbols $p_1$ and $p_2$, respectively, then the full symbol $\sigma(P_1P_2)$ is given by (see p.\,11 of  \cite{Taylor11.2}, or p.\,71 of \cite{Hormander85.3}, see also \cite{Grubb86,Treves80})
\begin{align*}
    \sigma(P_1P_2)\sim \sum_{J} \frac{1}{J !} \partial_{\xi}^{J} p_1 D_{x}^{J} p_2,
\end{align*}
where the sum is over all multi-indices $J$. Let $\sigma(B(x,\xi,\tau))\sim \sum_{j \leqslant -2} \textbf{\textit{b}}_j(x,\xi,\tau)$ be the full symbol of $B(x,\xi,\tau)$. We then get
\begin{align*}
    I_n & = \sum_{0=j+|J|+2-k} \frac{1}{J !} \partial_{\xi}^{J} \textbf{\textit{c}}_k D_{x}^{J} \textbf{\textit{b}}_{-2-j} = \textbf{\textit{c}}_2\textbf{\textit{b}}_{-2}, \\
    0 & = \sum_{1\leqslant l=j+|J|+2-k} \frac{1}{J !} \partial_{\xi}^{J} \textbf{\textit{c}}_k D_{x}^{J} \textbf{\textit{b}}_{-2-j} \\
    & = \textbf{\textit{c}}_2\textbf{\textit{b}}_{-2-l} + \sum_{\substack{l=j+|J|+2-k \\ j<l}} \frac{1}{J !} \partial_{\xi}^{J} \textbf{\textit{c}}_k D_{x}^{J} \textbf{\textit{b}}_{-2-j}, \quad l\geqslant 1.
\end{align*}
Therefore,
\begin{align*}
    \textbf{\textit{b}}_{-2}& = \textbf{\textit{c}}_2^{-1}, \\
    \textbf{\textit{b}}_{-2-l}& = -\textbf{\textit{c}}_2^{-1} \sum_{\substack{l=j+|J|+2-k \\ j<l}} \frac{1}{J !} \partial_{\xi}^{J} \textbf{\textit{c}}_k D_{x}^{J} \textbf{\textit{b}}_{-2-j}, \quad l\geqslant 1.
\end{align*}
In particular, by \eqref{3.1} we have
\begin{align*}
    \textbf{\textit{b}}_{-2} =
    \begin{pmatrix}
        \frac{1}{\tau - \mu|\xi|^2}I_n & 0 \\
        0 & \frac{1}{\tau - \alpha|\xi|^2}
    \end{pmatrix}
    + \frac{\lambda + \mu}{(\tau - \mu|\xi|^2)(\tau - (\lambda + 2\mu)|\xi|^2)}
    \begin{pmatrix}
        (g^{jm}\xi_k\xi_m) & 0 \\
        0 & 0
    \end{pmatrix},
\end{align*}
and
\begin{align}\label{3.4}
    \operatorname{Tr} \textbf{\textit{b}}_{-2} = \frac{n-1}{\tau - \mu|\xi|^2} + \frac{1}{\tau - \alpha|\xi|^2} + \frac{1}{\tau - (\lambda + 2\mu)|\xi|^2}.
\end{align}

\addvspace{10mm}

\section{Coefficients of the Asymptotic Expansions}\label{s3}

\addvspace{5mm}

\begin{proof}[Proof of Theorem {\rm\ref{thm1.1}}]

According to the theory of elliptic operators (see \cite{Morrey66,Morrey58.1,Morrey58.2,Pazy83,Stewart74}), we see that the thermoelastic operator $\mathcal{L}_g$ can generate strongly continuous semigroups $(e^{t\mathcal{L}_g^{\mp}})_{t\geqslant 0}$ with respect to the Dirichlet and Neumann boundary conditions, respectively. Moreover, there exist matrix-valued functions (integral kernels) $\textbf{\textit{K}}^{\mp}(t,x,y)$ such that (see \cite{Browder60} or p.\,4 of \cite{Fried64})
\begin{align*}
    e^{t\mathcal{L}_g^{\mp}} \textbf{\textit{V}}_0(x) = \int_{\Omega} \textbf{\textit{K}}^{\mp}(t,x,y)\textbf{\textit{V}}_0(y) \,dy, \quad \textbf{\textit{V}}_0(x)\in(L^2(\Omega))^{n+1}.
\end{align*}
Let $\{\textbf{\textit{U}}^{\mp}_k\}_{k \geqslant 1}$ be the orthonormal eigenvectors corresponding to the eigenvalues $\{\tau^{\mp}_k\}_{k \geqslant 1}$, then the integral kernels $\textbf{\textit{K}}^{\mp}(t,x,y)$ are given by
\begin{align*}
    \textbf{\textit{K}}^{\mp}(t,x,y)
    = e^{t\mathcal{L}_g^{\mp}} \delta(x - y)
    = \sum_{k=1}^{\infty} e^{-t\tau^{\mp}_k} \textbf{\textit{U}}^{\mp}_k(x) \otimes \textbf{\textit{U}}^{\mp}_k(y).
\end{align*}
This implies that
\begin{align*}
    \int_{\Omega} \operatorname{Tr} \textbf{\textit{K}}^{\mp}(t,x,y) \,dV
    = \sum_{k=1}^{\infty} e^{-t\tau^{\mp}_k}
\end{align*}
are spectral invariants.

We will combine calculus of symbols (see \cite{Seeley69}) and “method of images” (see \cite{Liu_NL21,McKeanSinger67}), to deal with the asymptotic expansions of traces of the integral kernels. Let $M = \Omega \cup \partial \Omega \cup \Omega^{*}$ be the closed double of $\Omega$, and $L_g$ the double to $M$ of the operator $\mathcal{L}_g$. The coefficients of the operator $L_g$ jump as $x$ crosses $\partial \Omega$, but $\frac{\partial \textbf{\textit{U}}}{\partial t} = L_g \textbf{\textit{U}}$ still has a nice fundamental solution (integral kernel) $\textbf{\textit{K}}(t,x, y)$ of class $C^{\infty}((0,+\infty) \times(M \setminus \partial \Omega) \times(M \setminus \partial \Omega)) \cap C^{1}((0,+\infty)\times M \times M)$, approximable even on $\partial \Omega$, and the integral kernels $\textbf{\textit{K}}^{\mp}(t,x,y)$ can be expressed on $(0,+\infty) \times \Omega \times \Omega$ as
\begin{align*}
    \textbf{\textit{K}}^{\mp}(t,x,y) = \textbf{\textit{K}}(t,x,y) \mp \textbf{\textit{K}}(t,x,y^*),
\end{align*}
where $y^*$ is the double of $y\in \Omega$. Since $(e^{tL_g})_{t\geqslant 0}$ can also be represented as
\begin{align*}
    e^{tL_g}
    = \frac{1}{2\pi i}\int_{\mathcal{C}} e^{-t\tau}(\tau I+L_g)^{-1} \,d\tau,
\end{align*}
where $\mathcal{C}$ is a suitable curve in the complex plane in the positive direction around the spectrum of $L_g$, that is, $\mathcal{C}$ is a contour around the positive real axis. It follows that, for any $x,y \in M$,
\begin{align*}
    \textbf{\textit{K}}(t,x,y) = e^{tL_g} \delta(x - y) = \frac{1}{(2\pi)^{n}} \int_{\mathbb{R}^n} e^{i (x-y)\cdot\xi} 
    \bigg(
        \frac{1}{2\pi i}\int_{\mathcal{C}} e^{-t\tau} \sigma\big((\tau I+L_g)^{-1}\big) \,d\tau
    \bigg) \,d\xi.
\end{align*}
In particular, for any $x\in \Omega$,
\begin{align*}
    \textbf{\textit{K}}(t,x,x) 
    & = \frac{1}{(2\pi)^{n}} \int_{\mathbb{R}^{n}}
    \bigg(
        \frac{1}{2\pi i}\int_{\mathcal{C}} e^{-t\tau} \sum_{l \geqslant 0} \textbf{\textit{b}}_{-2-l}\,d\tau 
    \bigg) \,d\xi, \\
    \textbf{\textit{K}}(t,x,x^*) 
    & = \frac{1}{(2\pi)^{n}} \int_{\mathbb{R}^{n}} e^{i (x-x^*)\cdot\xi}
    \bigg(
        \frac{1}{2\pi i}\int_{\mathcal{C}} e^{-t\tau} \sum_{l \geqslant 0} \textbf{\textit{b}}_{-2-l} \,d\tau 
    \bigg) \,d\xi.
\end{align*}
By applying the residue theorem to \eqref{3.4}, we get
    \begin{align*}
        \frac{1}{2\pi i}\int_{\mathcal{C}} e^{-t\tau} \operatorname{Tr} \textbf{\textit{b}}_{-2} \,d\tau 
        = (n-1)e^{-t\mu|\xi|^2} + e^{-t\alpha|\xi|^2} + e^{-t(\lambda + 2\mu)|\xi|^2}.
    \end{align*}
    We use the geodesic normal coordinate system centered at $x\in\Omega$, then $g_{ij}(x)=\delta_{ij}$ and $\Gamma^k_{ij}(x)=0$. It follows that
    \begin{align*}
        & \int_{\Omega}
        \biggl[\frac{1}{(2\pi)^{n}} \int_{\mathbb{R}^{n}}
        \bigg(
            \frac{1}{2\pi i}\int_{\mathcal{C}} e^{-t\tau} 
            \operatorname{Tr} \textbf{\textit{b}}_{-2}\,d\tau 
        \bigg) \,d\xi\biggr]\,dV 
        = \frac{\operatorname{vol}(\Omega)}{(4\pi t)^{n/2}}
        \biggl(\frac{n-1}{\mu^{n/2}} + \frac{1}{\alpha^{n/2}} + \frac{1}{(\lambda+2\mu)^{n/2}}\biggr).
    \end{align*}
    
    \addvspace{3mm}

    For given $\varepsilon > 0$, denote by $U_{\varepsilon}(\partial\Omega) = \{z\in M| \text{dist}(z,\partial\Omega) < \varepsilon\}$ the $\varepsilon$-neighborhood of $\partial\Omega$ in $M$.

    \addvspace{1mm}

    (i) For any $x\in\Omega \setminus U_{\varepsilon}(\partial\Omega)$, we have
    \begin{align*}
        &\frac{1}{(2\pi)^{n}} \int_{\mathbb{R}^{n}} e^{i(x-x^*)\cdot \xi}
        \bigg(
            \frac{1}{2\pi i}\int_{\mathcal{C}} e^{-t\tau} \operatorname{Tr} \textbf{\textit{b}}_{-2}\,d\tau
        \bigg) \,d\xi \\
        &\quad = \frac{1}{(4\pi t)^{n/2}} \biggl(\frac{n-1}{\mu^{n/2}} e^{-\frac{|x-x^*|^2}{4\mu t}} + \frac{1}{\alpha^{n/2}} e^{-\frac{|x-x^*|^2}{4\alpha t}} + \frac{1}{(\lambda+2\mu)^{n/2}} e^{-\frac{|x-x^*|^2}{4 (\lambda+2\mu) t}}\biggr),
    \end{align*}
    which tends to zero as $t\to 0^+$ since $|x-x^*| \geqslant \varepsilon$. Integrating the above result over $\Omega \setminus U_{\varepsilon}(\partial\Omega)$, we get $O(t^{1-n/2})$ as $t\to 0^+$.
    
    For $l \geqslant 1$, it can be verified that $\operatorname{Tr}\textbf{\textit{b}}_{-2-l}$ is a sum of finitely many terms, each of which has the form
    \begin{align*}
        \frac{r_k(x,\xi)}{(\tau - \mu|\xi|^2)^j (\tau - \alpha|\xi|^2)^m (\tau - (\lambda+2\mu)|\xi|^2)^p},
    \end{align*}
    where $k-2j-2m-2p=-2-l$, and $r_k(x,\xi)$ is the symbol independent of $\tau$ and homogeneous of degree $k$. Hence
    \begin{align*}
        \int_{\Omega}
        \biggl[\frac{1}{(2\pi)^{n}} \int_{\mathbb{R}^{n}}
        \bigg(
            \frac{1}{2\pi i}\int_{\mathcal{C}} e^{-t\tau} 
            \sum_{l\geqslant 1}\operatorname{Tr} \textbf{\textit{b}}_{-2-l} \,d\tau 
        \bigg) \,d\xi\biggr] dV
        & = O(t^{1-n/2}) \quad \text{as}\ t \to 0^+, \\
        \int_{\Omega}
        \biggl[\frac{1}{(2\pi)^{n}} \int_{\mathbb{R}^{n}} e^{i(x-x^*)\cdot \xi}
        \bigg(
            \frac{1}{2\pi i}\int_{\mathcal{C}} e^{-t\tau} 
            \sum_{l\geqslant 1}\operatorname{Tr} \textbf{\textit{b}}_{-2-l} \,d\tau 
        \bigg) \,d\xi\biggr] dV
        & = O(t^{1-n/2}) \quad \text{as}\ t \to 0^+.
    \end{align*}
    Combining the above results, we observe
    \begin{align*}
        \int_{\Omega} \operatorname{Tr}\textbf{\textit{K}}(t,x,x)\,dV 
        = \frac{\operatorname{vol}(\Omega)}{(4\pi t)^{n/2}}
        \biggl(\frac{n-1}{\mu^{n/2}} + \frac{1}{\alpha^{n/2}} + \frac{1}{(\lambda+2\mu)^{n/2}}\biggr) + O(t^{1-n/2}) \quad \text{as}\ t \to 0^+.
    \end{align*}

    \addvspace{1mm}

    (ii) We consider the case when $x\in\Omega \cap U_{\varepsilon}(\partial\Omega)$. Inspired by \cite{Liu_NL21,McKeanSinger67}, we pick a self-double patch $W$ of $M$ (such that $W \subset U_{\varepsilon}(\partial\Omega)$) covering a patch $W \cap \partial\Omega$ endowed with local coordinates $x$ such that $\varepsilon > x_n >0 $ in $W\cap\Omega$, $x_n=0$ on $W \subset U_{\varepsilon}(\partial\Omega)$, $x_n(x^*)=-x_n(x)$, and the positive $x_n$-direction is perpendicular to $\partial\Omega$ (see Figure \ref{fig1}). 
    
    \addvspace{-10pt}
    \begin{figure}[h]
        \centering
        \begin{tikzpicture}[scale=0.8,line width=1]
            \draw (0.03,0.5) arc (220:320:0.9);
            \draw (1.41,0.5) arc (135:45:1.4);
            \draw (3.39,0.5) arc (225:335:1.5);
            \draw (0.7,1.5) arc (135:45:2.5);
            \draw (0.7,1.5)--(1.8,-0.3);
            \draw (1.8,-0.3) arc (135:45:0.8);
            \draw (2.91,-0.3)--(4.22,1.5);
            \draw (2.91,-0.3)--(4.22,1.5);
            \draw (5.8,0.9) arc (-20:211: 3.2 and 2.5);
            \draw[->,>=stealth] (2.4,0.2) .. controls (2.4,1.3) and (2.4,2.3) .. (2.4,3.5);
            \draw[fill=black] (2.4,0.2) circle (1pt);
            \draw[fill=black] (2.4,1.65) circle (1pt);
            \node at (2.4,-0.5) {$\Omega^{*}$};
            \node at (3.2,3) {$x_n>0$};
            \node at (2.7,0.2) {$x^{*}$};
            \node at (2.7,1.65) {$x$};
            \node at (1.5,2.8) {$\Omega$};
            \node at (4.5,0.3) {$\partial\Omega$};
        \end{tikzpicture}

        \addvspace{-10pt}
        \caption{}\label{fig1}
    \end{figure}

    This has the effect that
    \begin{align*}
        g_{jk}(x^*) & =-g_{jk}(x) \quad \text{for}\ j<k=n\ \text{or}\ k<j=n,\\
        & =g_{jk}(x) \quad \text{for}\ j,k=n\ \text{or}\ j=k=n,\\
        g_{jk}(x) & =0 \quad \text{for}\ j<k=n\ \text{or}\ k<j=n\ \text{on}\ \partial\Omega,\\
        \sqrt{\operatorname{det}g/g_{nn}}\,dx_1\cdots dx_{n-1} &= \text{the element of Riemannian surface area on}\ \partial\Omega.
    \end{align*}
    We choose the geodesic normal coordinate system on $\partial\Omega$ (see \cite{LeeUhlm89} or p.\,532 of \cite{Taylor11.2}) such that the metric has the form $g = \sum_{\alpha,\beta=1}^{n-1} g_{\alpha\beta} \,dx_{\alpha} \,dx_{\beta} + dx_{n}^{2}$. Then we see that for $x\in \{z\in \Omega| \text{dist}(z,\partial\Omega) < \varepsilon\}$,
    \begin{align*}
        x-x^*=\big(0,\dots,0,x_n-(-x_n)\big)=(0,\dots,0,2x_n).
    \end{align*}
    We have
    \begin{align*}
        &\int_{W\cap\Omega}
        \biggl[\frac{1}{(2\pi)^{n}} \int_{\mathbb{R}^{n}} e^{i(x-x^*)\cdot \xi}
        \bigg(
            \frac{1}{2\pi i}\int_{\mathcal{C}} e^{-t\tau} 
            \operatorname{Tr} \textbf{\textit{b}}_{-2} \,d\tau 
        \bigg) \,d\xi\biggr] dV \\
        &\quad = \frac{1}{(4\pi t)^{n/2}} \int_0^{\varepsilon} dx_n \int_{W\cap\partial\Omega}
        \biggl(
            \frac{n-1}{\mu^{n/2}}e^{-\frac{x_n^2}{\mu t}} + \frac{1}{\alpha^{n/2}}e^{-\frac{x_n^2}{\alpha t}}  + \frac{1}{(\lambda+2\mu)^{n/2}}e^{-\frac{x_n^2}{(\lambda+2\mu) t}}
        \biggr)
        dx^{\prime} \\
        &\quad = \frac{1}{(4\pi t)^{n/2}} \int_0^{+\infty} dx_n \int_{W\cap\partial\Omega}
        \biggl(
            \frac{n-1}{\mu^{n/2}}e^{-\frac{x_n^2}{\mu t}} + \frac{1}{\alpha^{n/2}}e^{-\frac{x_n^2}{\alpha t}}  + \frac{1}{(\lambda+2\mu)^{n/2}}e^{-\frac{x_n^2}{(\lambda+2\mu) t}}
        \biggr)
        dx^{\prime} \\
        &\qquad -\frac{1}{(4\pi t)^{n/2}} \int_{\varepsilon}^{+\infty} dx_n \int_{W\cap\partial\Omega}
        \biggl(
            \frac{n-1}{\mu^{n/2}}e^{-\frac{x_n^2}{\mu t}} + \frac{1}{\alpha^{n/2}}e^{-\frac{x_n^2}{\alpha t}}  + \frac{1}{(\lambda+2\mu)^{n/2}}e^{-\frac{x_n^2}{(\lambda+2\mu) t}}
        \biggr)
        dx^{\prime} \\
        &\quad = \frac{1}{4} \frac{\operatorname{vol}(W\cap \partial \Omega)}{(4\pi t)^{(n-1)/2}} \biggr(\frac{n-1}{\mu^{(n-1)/2}} + \frac{1}{(\lambda+2\mu)^{(n-1)/2}} + \frac{1}{\alpha^{(n-1)/2}}\biggr)  \\
        &\qquad - \frac{1}{(4\pi t)^{n/2}} \int_{W\cap\partial\Omega} dx^{\prime} \int_{\varepsilon}^{+\infty}  
        \biggl(
            \frac{n-1}{\mu^{n/2}}e^{-\frac{x_n^2}{\mu t}} + \frac{1}{\alpha^{n/2}}e^{-\frac{x_n^2}{\alpha t}} + \frac{1}{(\lambda+2\mu)^{n/2}}e^{-\frac{x_n^2}{(\lambda+2\mu) t}}
        \biggr)dx_n.
    \end{align*}
    Note that, for any fixed $\varepsilon >0$, the last integral is $O(t^{1-n/2})$ as $t \to 0^+$. Hence
    \begin{align*}
        &\int_{W\cap\Omega}
        \biggl[\frac{1}{(2\pi)^{n}} \int_{\mathbb{R}^{n}} e^{i(x-x^*)\cdot \xi}
        \bigg(
            \frac{1}{2\pi i}\int_{\mathcal{C}} e^{-t\tau} 
            \operatorname{Tr} \textbf{\textit{b}}_{-2} \,d\tau 
        \bigg) \,d\xi\biggr] dV \\
        &\quad = \frac{1}{4} \frac{\operatorname{vol}(W\cap \partial \Omega)}{(4\pi t)^{(n-1)/2}} \biggr(\frac{n-1}{\mu^{(n-1)/2}} + \frac{1}{(\lambda+2\mu)^{(n-1)/2}} + \frac{1}{\alpha^{(n-1)/2}}\biggr) +O(t^{1-n/2})\quad {\rm as}\ t \to 0^+.
    \end{align*}
    Similarly, we obtain
    \begin{align*}
        \int_{W\cap\Omega} \operatorname{Tr}\textbf{\textit{K}}(t,x,x^*)\,dV 
        =& \frac{1}{4} \frac{\operatorname{vol}(W\cap \partial \Omega)}{(4\pi t)^{(n-1)/2}} \biggr(\frac{n-1}{\mu^{(n-1)/2}} + \frac{1}{(\lambda+2\mu)^{(n-1)/2}} + \frac{1}{\alpha^{(n-1)/2}}\biggr) \\
        &+O(t^{1-n/2})\quad {\rm as}\ t \to 0^+.
    \end{align*}
    Therefore
    \begin{align*}
        &\int_{W\cap\Omega} \operatorname{Tr}\textbf{\textit{K}}^{\mp}(t,x,x)\,dV
        = \int_{W\cap\Omega} \operatorname{Tr}\textbf{\textit{K}}(t,x,x)\,dV \mp \int_{W\cap\Omega} \operatorname{Tr}\textbf{\textit{K}}(t,x,x^*)\,dV \\
        &\quad = \frac{\operatorname{vol}(W\cap \Omega) }{(4\pi t)^{n/2}} \biggr(\frac{n-1}{\mu^{n/2}} + \frac{1}{(\lambda+2\mu)^{n/2}} + \frac{1}{\alpha^{n/2}}\biggr) \\
        &\qquad \mp \frac{1}{4} \frac{\operatorname{vol}(W\cap \partial \Omega)}{(4\pi t)^{(n-1)/2}} \biggr(\frac{n-1}{\mu^{(n-1)/2}} + \frac{1}{(\lambda+2\mu)^{(n-1)/2}} + \frac{1}{\alpha^{(n-1)/2}}\biggr) + O(t^{1-n/2}) \quad \text{as}\ t \to 0^+.
    \end{align*}
    Similarly, we can also get the other coefficients $a_k^{\mp}$ for $2 \leqslant k \leqslant n$ by this method.
\end{proof}

\addvspace{5mm}
\begin{proof}[Proof of Corollary {\rm \ref{cor1.4}}]
    Theorem {\rm \ref{thm1.1}} shows that $a_0$ and $a_1^{\mp}$ are spectral invariants. Since the thermoelastic spectrum of $\Omega$ with respect to the Dirichlet or Neumann boundary condition is equal to that of the geodesic ball $B^n\subset \mathcal{M}$, we observe that
    \begin{align}\label{4.1}
        \operatorname{vol}(\Omega)=\operatorname{vol}(B^n),\quad \operatorname{vol}(\partial\Omega)=\operatorname{vol}(\partial B^n).
    \end{align}
    In addition, the isoperimetric inequality holds on $\mathcal{M}$ by condition $(b)$. Then
    \begin{align}\label{4.2}
        \operatorname{vol}(\partial \Omega) \geqslant \operatorname{vol}(\partial B^n),
    \end{align}
    equality holds if and only if $\Omega$ is isometric to $B^n$. Therefore, we immediately deduce that $\Omega$ is isometric to $B^n$ by combining \eqref{4.1} and \eqref{4.2}.
\end{proof}

\addvspace{10mm}

\section*{Acknowledgements}

\addvspace{5mm}

This research was supported by NNSF of China (11671033/A010802) and NNSF of China \\
(11171023/A010801).

\addvspace{10mm}

\addvspace{5mm}

\end{document}